\documentclass[10.5 pt]{article} \hoffset=-1cm \topmargin=-1.5cm

\usepackage{latexsym}
\usepackage{amsthm}
\usepackage{amssymb}
\usepackage{amsmath,float}
\usepackage{tikz}

\newtheorem{theorem}{Theorem}[section]

\newtheorem{thm}[theorem]{Theorem}

\newtheorem{corollary}[theorem]{Corollary}

\newtheorem{problem}{Problem}

\newtheorem{remark}[theorem]{Remark}

\theoremstyle{definition}

\title{\bf Some Criteria for a Signed Graph to Have Full Rank }
\author{\bf {S. Akbari\footnote{Email addresses: s\underline{~}akbari@sharif.edu, ghafaribaghestani\underline{~}a@mehr.sharif.edu, kazemian\underline{~}kimia@mehr.sharif.edu, nahvi\underline{~}mina@mehr.sharif.edu}, A. Ghafari, K. Kazemian, M. Nahvi}\\\\
	\small Department of Mathematical Sciences\\
	 \small Sharif University of Technology, Tehran, Iran\\}

\date{}
\begin{document}
	\maketitle
	\begin{abstract}
		A weighted graph $G^{\omega}$ consists of a simple graph $G$ with a weight $\omega$, which is a mapping,\linebreak$\omega$:  $E(G)\rightarrow\mathbb{Z}\backslash\{0\}$.
		 A signed graph is a graph whose edges are labeled with $-1$ or $1$. In this paper, we characterize graphs which have a sign such that their signed adjacency matrix has full rank, and graphs which have a weight such that their weighted adjacency matrix does not have full rank. We show that for any arbitrary simple graph $G$, there is a sign $\sigma$ so that $G^{\sigma}$ has full rank if and only if $G$ has a $\{1,2\}$-factor. We also show that for a graph $G$, there is a weight $\omega$ so that $G^{\omega}$ does not have full rank if and only if $G$ has at least two $\{1,2\}$-factors.
	\end{abstract}
	{\noindent\it\bf 2010 Mathematics Subject Classification:} 05C22, 05C70, 05C78, 15A03.\\
	{\it\bf Keywords:}  Weighted graph, Signed graph, Weighted adjacency matrix, Signed adjacency matrix, Rank.
	\section{Introduction}
	Throughout this paper, by a graph we mean a simple, undirected and finite graph. Let $G$ be a graph. We denote the edge set and the vertex set of $G$ by $E(G)$ and $V(G)$, respectively. By \textit{order} and \textit{size} of $G$, we mean the number of vertices and the number of edges of $G$, respectively. The adjacency matrix of a simple graph $G$ is denoted by $A(G)=[a_{ij}]$, where $a_{ij}=1$ if $v_{i}$ and $v_{j}$ are adjacent, and $a_{ij}=0$ otherwise. We denote the complete graph of order $n$ by $K_{n}$. A \textit{$\{1,2\}$-factor} of a graph $G$ is a spanning subgraph of $G$ which is a disjoint union of copies of $K_{2}$ and cycles. For a $\{1,2\}$-factor $H$ of $G$, the number of cycles of $H$ is denoted by $c(H)$. The number of $\{1,2\}$-factors of a graph $G$ is denoted by $t(G)$. The \textit{perrank} of a graph $G$ of order $n$ is defined to be the order of its largest subgraph which is a disjoint union of copies of $K_{2}$ and cycles, and we say $G$ has full perrank if $perrank(G)=n$. For a graph $G$, a \textit{zero-sum flow} is an assignment of non-zero real numbers to the edges of $G$ such that the total sum of the assignments of all edges incident with any vertex  is zero. For a positive integer $k$, a \textit{zero-sum $k$-flow} of $G$ is a zero-sum flow of $G$ using the numbers $\{\pm1,\ldots,\pm(k-1)\}$.\\
	We call a matrix \textit{integral} if all of its entries are integers. A set $X$ of $n$ entries of an $n\times n$ matrix $A$ is called a \textit{transversal}, if $X$ contains exactly one entry of each row and each column of $A$. A transversal is called a \textit{non-zero transversal} if all its entries are non-zero. The identity matrix is denoted by $I$. Also, $j_{n}$ is an $n\times1$ matrix with all entries 1.\\
A \textit{weighted graph} $G^{\omega}$ consists of a simple graph $G$ with a weight $\omega$, which is a mapping,\linebreak $\omega$: $ E(G)\rightarrow\mathbb{Z}\backslash\{0\}$. A \textit{signed graph} $G^{\sigma}$ is a weighted graph where $\sigma$: $E(G)\rightarrow\{-1,1\}$. The \textit{weighted adjacency matrix} of the weighted graph $G^{\omega}$ is denoted by $A(G^{\omega})=[a^{\omega}_{ij}]$, where $a^{\omega}_{ij}=\omega(v_{i}v_{j})$ if $v_{i}$ and $v_{j}$ are adjacent vertices, and $a^{\omega}_{ij}=0$, otherwise. The rank of a weighted graph is defined to be the rank of its weighted adjacency matrix. 
A \textit{bidirected} graph $G$ is a graph such that each edge is composed of two directed \textit{half edges}. Function $f:$ $E(G)\rightarrow\mathbb{Z}\backslash\{0\}$ is a \textit{nowhere-zero $\mathbb{Z}$-flow} of $G$ if for every vertex $v$ of $G$ we have $\sum_{e\in E^{+}(v)}{f(e)}=\sum_{e\in E^{-}(v)}{f(e)}$, where $E^{+}(v)$ (resp. $E^{-}(v)$) is the set of all edges with tails (resp. heads) at $v$. For a positive integer $k$, a \textit{nowhere-zero $k$-flow} of $G$ is a nowhere-zero $\mathbb{Z}$-flow of $G$ using the numbers $\{\pm1,\ldots,\pm(k-1)\}$.
For a graph $G$, where $E(G)=\{e_{1},\ldots,e_{m}\}$ and $V(G)=\{v_{1},\ldots,v_{n}\}$, we define $M_{G}(x_{1},\ldots,x_{m})=[m_{ij}]$ to be an $n\times n$ matrix, where $m_{ij}=\begin{cases}
x_{k} & \text{If } e_{k}=v_{i}v_{j}
\\
0 & \text{otherwise}
\end{cases} $. Now, define  $f_{G}(x_{1},\ldots,x_{m})=det(M_{G}(x_{1},\ldots,x_{m}))$.\\
\\
	 In order to establish our results, first we need the following well-known theorem, which has many applications in algebraic combinatorics.\\
	 \\
	 \noindent\textbf{Theorem A.} \cite{[1]} \textit{Let $F$ be an arbitrary field and let $f=f(x_{1},\ldots,x_{n})$ be a polynomial in $F[x_{1},\ldots,x_{n}]$. Suppose the degree $deg(f)$ of f is $\sum\limits_{i=1}^{n}t_{i}$, where each $t_{i}$ is a nonnegative integer, and suppose the coefficient of $\prod_{i=1}^{n} x^{t_{i}}_{i}$ in $f$ is non-zero. Then, if $S_{1},\ldots,S_{n}$ are subsets of $F$ with $|S_{i}|>t_{i}$, there are $s_{1}\in S_{1},s_{2}\in S_{2},\ldots,s_{n}\in S_{n}$ so that $f(s_{1},\ldots,s_{n})\neq0$.}\\
	 \\
	 The following remark is a necessary tool in proving our results.\\
	 \begin{remark}
	 	\label{remark1}
	 	For a graph $G$, each non-zero transversal in $A(G)$ corresponds to a $\{1,2\}$-factor of $G$, and each $\{1,2\}$-factor $H$ of $G$ corresponds to $2^{c(H)}$ non-zero transversals in $M$.
	 \end{remark}
	 In this paper, we prove the following theorems: 
	  \vspace{0.2cm}
	 \\
	 \noindent\textbf{Theorem.}\textit{
	 	Let $G$ be a graph. Then there exists a sign $\sigma$ for $G$ so that $G^{\sigma}$ has full rank if and only if $G$ has full perrank.}\\
	 \\
	\noindent\textbf{Theorem.}\textit{
	 			Let $G$ be a graph. Then there exists a weight $\omega$ for $G$ so that $G^{\omega}$ does not have full rank if and only if $t(G)\geq2$.}\\
		\section{Signed Graphs Which Have Full Rank}
		Our first result is the following theorem.\\
		\begin{thm}
			Let $G$ be a graph. Then there exists a sign $\sigma$ for $G$ so that $G^{\sigma}$ has full rank if and only if $G$ has full perrank.
			\end{thm}
		\begin{proof}
			First, assume that $G$ has full perrank. Let $m=|E(G)|$, $n=|V(G)|$. Suppose that $\overline{f}(x_{1},\ldots,x_{m})$ is the polynomial obtained by replacing $x^{2}_{i}$ by $1$ in $f_{G}(x_{1},\ldots,x_{m})$, for all $i$, $1\leq i\leq m$. Clearly, $deg_{x_{i}}\overline{f}\leq 1$ for all $i$, $1\leq i\leq m$.\\
			Let $U$ be the set of all $\{1,2\}$-factors of $G$. We have $U\neq\varnothing$. For each $H\in U$, we define $a(H)$ as the number of $K_{2}$ components of $H$. We choose $F_{1}$ in $U$ so that $a(F_{1})=\max\limits_{H\in U}(a(H))$. Note that $F_{1}$ does not contain any even cycles. So $F_{1}$ is a disjoint union of $a(F_{1})$ copies of $K_{2}$, and $c(F_{1})$ odd cycles. Assume that all edges in the cycles of $F_{1}$ are $e_{1},\ldots,e_{k}$. There are $2^{c(F_{1})}$ non-zero transversals in $M$ corresponding to $F_{1}$, and all terms in $\overline{f}(x_{1},\ldots,x_{m})$ created by these non-zero transversals are $(-1)^{a(F_{1})}x_{1}\ldots x_{k}$.\\
			Let $X$ be a non-zero transversal in $M$ which makes the term $ax_{1} \ldots x_{k}$ in $\overline{f}(x_{1},\ldots,x_{m})$ for some $a\in \mathbb{R}$ and does not correspond to $F_{1}$. Let $F_{2}$ be the $\{1,2\}$-factor of $G$ associated with $X$. The edges in the cycles of $F_{2}$ are $e_{1},\ldots,e_{k}$. Therefore $a(F_{2})=a(F_{1})$, which is maximum. So $F_{2}$ has no even cycles and $a=(-1)^{a(F_{1})}$. So $x_{1} \ldots x_{k}$ has a non-zero coefficient in $\overline{f}(x_{1},\ldots,x_{m})$. Thus $\overline{f}\not\equiv0$ and it has a term $c\prod_{i=1}^{m} x^{t_{i}}_{i}$ with the maximum degree $\sum\limits_{i=1}^{m}t_{i}$, where $t_{i}\in \{0,1\}$ for each $i$, $1\leq i\leq m$. Let $S_{i}=\{-1,1\}$ for each $i$, $1\leq i\leq m$, so $|S_{i}|\geq2$. By Theorem A there exists $(s_{1},\ldots,s_{m})\in S_{1}\times \cdots \times S_{m}$ so that $\overline{f}(s_{1},\ldots,s_{m})\neq0$. By defining $\sigma$ as a sign assigning the same sign as $s_{i}$ to $e_{i}$ for each $i$, $1\leq i\leq m$, one can see that $det(A(G^{\sigma}))\neq0$.\\
			
			Now, assume that there exists such a sign for $G$. It is clear that $A(G^{\sigma})$ has a non-zero transversal, and the associated edges with this transversal, regardless of their signs, form a $\{1,2\}$-factor of $G$.
		\end{proof}
		\begin{corollary}
				Let $G$ be a graph. Then $\max\limits_{\sigma}(rank(G^{\sigma}))=perrank(G)$.
		\end{corollary}
		 In the sequel we propose two following problems.
		\begin{problem}
			Let $G$ be a graph. Determine $\min\limits_{\sigma}(rank(G^{\sigma}))$.
		\end{problem}
		\begin{problem}
			Find an efficient algorithm that can lead us to the desirable sign in Theorem $2.1$.
		\end{problem}
		\section{Weighted Graphs Which Do Not Have Full Rank}
		
		In order to establish our next result, first we need the following lemma and theorems.\\
		
		\vspace{0.3 cm}
		\noindent\textbf{Lemma A.} \cite{[7]} \textit{Let $G$ be a $2$-edge connected bipartite graph. Then $G$ has a zero-sum $6$-flow.}\\
		
		\vspace{0.3 cm}
		\noindent\textbf{Theorem B.} \cite{[7]} \textit{Suppose $G$ is not a bipartite graph. Then $G$ has a zero-sum flow if and only if for any edge $e$ of $G$, $G\backslash \{e\}$ has no bipartite component.}\\
			
			\vspace{0.3 cm}
			\noindent\textbf{Theorem C.} \cite{[8]} \textit{Every bidirected graph with a nowhere-zero $\mathbb{Z}$-flow has a nowhere-zero $12$-flow.}\\
				
				\vspace{0.3 cm}
					According to \cite{[2]}, if we orient all edges of a simple graph in a way that all edges adjacent to each vertex $v$ belong to $E^{+}(v)$, then a nowhere-zero bidirected flow corresponds to a zero-sum flow. Therefore, the following corollary is a result of Theorem C.\\
					
					\vspace{0.3 cm}
					\noindent\textbf{Corollary A.} \textit{Every graph with a zero-sum flow has a zero-sum $12$-flow.}\\
					\\
		Now, we can prove the following theorem.\\
		\begin{thm}
			Let $G$ be a graph. Then there exists a weight $\omega$ for $G$ so that $G^{\omega}$ does not have full rank if and only if $t(G)\geq2$.
		\end{thm}
		\begin{proof}
			Define $X=\{H|$ $t(H)\geq2$ and for any $\omega,$ $rank(A(H^{\omega}))=|V(H)|\}$. By contradiction assume that $X\not=\varnothing$. Let $n=\min\limits_{H\in X}{|V(H)|}$ and $m=\min\limits_{H\in X,|V(H)|=n}{|E(H)|}$ and $G\in X$ be a graph of order $n$ and size $m$. Let $E(G)=\{e_{1},\ldots,e_{m}\}$. Obviously, $G$ is connected. For all $a_{1},\ldots,a_{m}\in \mathbb{Z}\backslash \{0\}$, we have $f_{G}(a_{1},\ldots,a_{m})\not=0$. For each $i$, $1\leq i\leq m$, we can write $f_{G}=x_{i}^{2}g_{G_{i}}+x_{i}h_{G_{i}}+l_{G_{i}}$, where $g_{G_{i}}$, $h_{G_{i}}$ and $l_{G_{i}}$ are polynomials in variables $x_{1},\ldots,x_{i-1},x_{i+1},\ldots,x_{m}$. One can see that $g_{G_{i}}$ is the zero polynomial if and only if $e_{i}$ belongs to no $K_{2}$ component of any $\{1,2\}$-factor of $G$, $h_{G_{i}}$ is the zero polynomial if and only if $e_{i}$ belongs to no cycle of any $\{1,2\}$-factor of $G$, and $l_{G_{i}}$ is the zero polynomial if and only if $e_{i}$ belongs to all $\{1,2\}$-factors of $G$.\\
			If there exists an edge $e_{i}$ of $G$ belonging to no $\{1,2\}$-factor of $G,$ then $t(G\backslash\{e_{i}\})\geq2$ and for each weight $\omega$ we have $det(A((G\backslash\{e_{i}\})^{\omega}))\not=0$, which is a contradiction. So each edge of $G$ is contained in at least one $\{1,2\}$-factor. Moreover, if there exists an edge $e_{i}=uv$ which belongs to no cycle of any $\{1,2\}$-factor of $G$ but belongs to all $\{1,2\}$-factors of $G$, then we have $t(G\backslash\{u,v\})\geq2$ and for each weight $\omega$, $det(A((G\backslash\{u,v\})^{\omega}))\not=0$, a contradiction. Furthermore we show that if $e_{i}$ appears in a cycle of a $\{1,2\}$-factor, then neither of the polynomials $g_{G_{i}}$ and $l_{G_{i}}$ is the zero polynomial. By contradiction assume that $g_{G_{i}}l_{G_{i}}\equiv0$. If $g_{G_{i}}\equiv0$ and $l_{G_{i}}\not\equiv0$, then according to Theorem A there are $a_{1},\ldots,a_{i-1},a_{i+1},\ldots,a_{m}\in \mathbb{Z}\backslash\{0\}$ so that $h_{G_{i}}l_{G_{i}}(a_{1},\ldots,a_{i-1},a_{i+1},\ldots,a_{m})\not=0$. It can be seen that $f_{G}(a_{1},\ldots,a_{m})=0$, where $a_{i}=-\dfrac{l_{G_{i}}}{h_{G_{i}}}(a_{1},\ldots,a_{i-1},a_{i+1},\ldots,a_{m})$. Note that for each $j$, $1\leq j\leq m$, $a_{j}$ is a non-zero rational number. Since $f_{G}$ is a homogeneous polynimial, it has a root in $(\mathbb{Z}\backslash\{0\})^{m}$, a contradiction. If $l_{G_{i}}\equiv0$ and $g_{G_{i}}\not\equiv0$, then the same argument leads to a  contradiction. If $g_{G_{i}}=l_{G_{i}}\equiv0$ and 
			$e_i,e_{j_1},\ldots,e_{j_p}$ are all edges of a cycle $C$ of a $\{1,2\}$-factor of $G$, then for each $k$, $1\leq k \leq p$, we have $h_{G_i}=x_{j_k} p_k +q_k$, where $p_k$ and $q_k$ are polynomials in variables $\{x_j\}_{j \in I}$, where $I=\{1,\ldots,n\} \backslash \{i,j_{k}\}$. Obviously, $p_k \not\equiv 0$ . If $q_k \not\equiv 0$, then using Theorem A one can see that $h_{G_i}$ has a root in $(\mathbb{Z}\backslash \{0\})^{m-1}$, a contradiction. Therefore, one can see that $h_{G_{i}}=x_{j_{1}}\cdots x_{j_{p}}h_{1}$, where $h_{1}$ is a polynomial in variables $\{x_{j}\}_{j\in J}$, where $J=\{1,\ldots,m\}\backslash\{i,j_{1},\ldots,j_{p}\}$. So $C$ is a subgraph of every $\{1,2\}$-factor of $G$. Now, by considering the graph $G\backslash V(C)$ and noting that $t(G\backslash V(C))\geq2$, we obtain a contradiction. So we have $g_{G_{i}},l_{G_{i}}\not\equiv0$ for all $i$, $1\leq i\leq m$. So, for each edge $e_i$ of $G$, there exists a $\{1,2\}$-factor of $G$ not containing $e_i$, and also there exists a $\{1,2\}$-factor of $G$ containing $e_i$ in a $K_2$ component. \\
			If $G$ has a zero-sum flow, then $M_{G}(a_{1},\ldots,a_{m})j_{n}=0$, and as a result $f_{G}(a_{1},\ldots,a_{m})=0$, where for each $i$, $1\leq i\leq m$, $a_{i}$ is the non-zero integer assigned to the edge $e_{i}$ in the flow, a contradiction. Thus, assume that $G$ has no zero-sum flow. Now, we have two cases:
			\begin{enumerate}
					\item The graph $G$ is not bipartite. According to Theorem B, $G\backslash\{e_{i}\}$ has a bipartite component for some $i$, $1\leq i\leq m$. We have two cases:
					\begin{enumerate}
						\item The edge $e_{i}$ is not a cut edge. The graph $G\backslash\{e_{i}\}$ is a bipartite graph, say $G\backslash\{e_{i}\}=(X,Y)$, where the vertices adjacent to $e_{i}$ belong to $X$. There is a $\{1,2\}$-factor of $G$ having $e_{i}$ in a $K_{2}$ component, so $|X|-2=|Y|$. Also, there exists a $\{1,2\}$-factor of $G$ not having $e_{i}$, therefore we have $|X|=|Y|$, a contradiction.
						\item\label{item1} Now, assume that $e_{i}$ is a cut edge. The graph $G\backslash\{e_{i}\}$ has two components $H$ and $F$, where $F=(X,Y)$ is bipartite.
					\end{enumerate}
				\item\label{item2} Now, suppose that $G$ is bipartite. According to Lemma A, $G$ has a cut edge $e_{i}$. We denote the bipartite connected components of the graph $G\backslash\{e_{i}\}$ by $H$ and $F=(X,Y)$.
			\end{enumerate}
			In both Cases \ref{item1} and \ref{item2}, there exists a $\{1,2\}$-factor of $G$ having $e_{i}$ in a $K_{2}$ component. Therefore, $F\backslash u$ has a $\{1,2\}$-factor, where $u\in X$ is the vertex in $F$ adjacent to $e_{i}$, so we have $|X|=|Y|+1$. On the other hand, there exists a $\{1,2\}$-factor of $G$ which does not contain $e_{i}$. Hence, $F$ has a $\{1,2\}$-factor. So we 
			have $|X|=|Y|$, a contradiction.\\
			
			Now, let $G^{w}$ be a weighted graph which has full rank. Let $E(G)=\{e_1,\ldots,e_m\}$. By contradiction assume that $t(G) < 2$. Then $f_{G}$ has one monomial and therefore for some $i$, $1\leq i \leq m$, $w(e_i)=0$, a contradiction. 
		\end{proof}
		\begin{remark}
			Let $G$ be a graph with $t(G)\geq2$. According to Lemma A, if $G$ is bipartite, then there exists a weight $\omega$: $E(G)\rightarrow\{\pm1,\ldots,\pm5\}$ such that $G^{\omega}$ does not have full rank. If $G$ is not bipartite, then according to Corollary A, there exists a weight $\omega$: $E(G)\rightarrow\{\pm1,\ldots,\pm11\}$ such that $G^{\omega}$ does not have full rank.
	    \end{remark}
	    \noindent{\bf Acknowledgement.} The authors are deeply grateful to Mohammad Javad Moghadamzadeh for his fruitful comments in the preparation of this paper.
			
	\end{document}